\documentclass{amsart}
\usepackage{amssymb}
\usepackage{amsmath,amscd}
\usepackage{combelow}
\usepackage{geometry}
\newcommand{\N}{\mathbb{N}}

\newcommand{\II}{\mathcal{I}}

\newcommand{\G}{\mathcal{G}}

\newcommand{\wt}{\widetilde}

\newcommand{\arr}{\rightarrow}

\newcommand{\fii}{\varphi}

\newcommand{\vv}{\vert\vert}
\newcommand{\vvh}{\vert\vert_{Hof}}

\newcommand{\Homeo}{\operatorname{Homeo}}
\newcommand{\Diff}{\operatorname{Diff}}

\newcommand{\Poiss}{\operatorname{Poiss}(M, \pb)}
\newcommand{\Ham}{\operatorname{Ham}(M, \pb)}

\newcommand{\ccon}{\overset{C^0}{\longrightarrow}}
\newcommand{\hcon}{\overset{Hof}{\longrightarrow}}

\newtheorem{theorem}{Theorem}
\newtheorem{proposition}[theorem]{Proposition}
\newtheorem{corollary}[theorem]{Corollary}

\newtheorem{remark}{Remark}

\newtheorem*{theorem*}{Theorem}
\newtheorem*{remark*}{Remark}
\newtheorem*{definition*}{Definition}

\newtheorem{question}{Question}
\newtheorem*{question*}{Question}

\theoremstyle{remark}

\newcommand{\X}{\mathfrak{X}}

\newcommand{\pb}{\{ \cdot, \cdot \}}

\title{$C^0$-rigidity of Poisson diffeomorphisms}
\author{Du\v{s}an Joksimovi\'c}

\begin{document}

\maketitle

\vspace{-6ex}
\begin{abstract}
We prove the Poisson version of the Eliashberg-Gromov $C^0$-rigidity. More precisely, we prove that the group of Poisson diffeomorphisms is closed with respect to the $C^0$ topology inside the group of all diffeomorphisms. The proof relies on the Poisson version of the energy-capacity inequality.
\end{abstract}

\section{Introduction and main result} \label{sec:intro}

The famous $C^0$-rigidity theorem of Y.~Eliashberg and M.~Gromov states that the group of symplectic diffeomorphisms forms a closed subset of the group of all diffeomorphisms equipped with the $C^0$ topology (i.e.~the compact-open topology). This result led to definitions of symplectic homeomorphisms and topological symplectic manifolds, and is considered as the beginning of the subfield called $C^0$-symplectic geometry. Roughly, it investigates non-smooth symplectic objects and the behaviour of smooth symplectic objects with respect to the $C^0$ topology. The aim of this article is to extend this philosophy to Poisson manifolds. As the main result, we will prove the Poisson analog of the Eliashberg-Gromov theorem.

A Poisson structure on a smooth manifold $M$ is a Lie bracket $\pb$ on the space $C^{\infty}(M)$ which satisfies the Leibniz identity
\begin{equation} \label{eq:leibniz}
    \{fg,h\} = f\{g,h\}+ g\{f,h\}, \quad \forall f,g,h \in C^{\infty}(M).
\end{equation}
Note that every manifold carries the trivial Poisson structure $\pb \equiv 0.$ Moreover, examples of Poisson manifolds include all symplectic manifolds and duals of Lie algebras.
\begin{definition*} \label{def:poisson-diff}
\emph{A Poisson diffeomorphism} is a smooth diffeomorphism $\psi: (M, \pb) \arr (M,\pb)$ that satisfies
\begin{equation} \label{eq:poisson}
    \{f,g\} \circ \psi = \{f \circ \psi, g \circ \psi\}, 
\end{equation}
for all $f,g \in C^{\infty}(M).$
We denote the group of all Poisson diffeomorphisms by $\Poiss.$
\end{definition*}

\begin{remark} \label{rmk:comp_supp} \normalfont
Notice that being a Poisson map is a local condition. Therefore to check that a diffeomorphism $\psi: (M, \pb) \arr (M,\pb)$ is Poisson it is enough to verify condition \eqref{eq:poisson} for all compactly supported functions $f$ and $g.$ We will exploit this fact in the proof of Theorem \ref{main-thm}.
\end{remark}

Let $(M,\pb)$ be a Poisson manifold and let $f \in C_c^{\infty}([0,1] \times M)$ be a compactly supported time-dependent Hamiltonian function. We define the (time-dependent) Hamiltonian vector field $X_f^t$ associated to $f$ by
$$X_f^t := \{f_t, \cdot\} \in \X (M),$$
where $f_t := f(t, \cdot) \in C^{\infty}(M),$ $t\in [0,1].$ \footnote{Note that the Leibniz identity \eqref{eq:leibniz} implies that $\{f_t, \cdot\}$ is a derivation of $C^{\infty}(M)$ and as such it defines a (time-dependent) vector field on $M.$}
The flow $\{\varphi_f^t\}$  of $X_f^t$ is called the Hamiltonian flow (or the Hamiltonian isotopy) generated by $f.$ 
The Hamiltonian group of $(M,\pb)$ is
\[\mathrm{Ham}(M,\pb):=\big\{\varphi^1_f \hspace{1mm} \vert \hspace{1mm} f\in C^{\infty}_{c}([0,1]\times M)\big\}.\]
The orbits of the standard action of $\Ham$ on $M$ induce a foliation of $M$ which is called the \emph{symplectic foliation} and its leaves are called \emph{symplectic leaves}.

Although the condition of being a Poisson map includes the derivative of the map it turns out that Poisson diffeomorphisms behave nicely with respect to the $C^0$ limits.
The main result of this article is the following.

\begin{theorem}[$C^0$ rigidity of Poisson diffeomorphisms] \label{main-thm}
Let $(M,\pb)$ be a Poisson manifold. Then $\Poiss$ is a closed subset of $\Diff(M)$ with respect to the $C^0$ topology.
\end{theorem}

Theorem \ref{main-thm} generalizes the Eliashberg-Gromov theorem to general Poisson manifolds, in the sense that any symplectic manifold $(M,\omega)$ carries the natural Poisson structure 
$\{\cdot,\cdot\} := \pb_{\omega}$ given by 
$$\{f_1, f_2\} := \omega (X_{f_1}, X_{f_2}), \quad \forall f_1,f_2 \in C^{\infty} (M),$$
where $X_{f_i}$ are the corresponding Hamiltonian vector fields (i.e.~unique vector fields  that satisfy $df_i = \omega(X_{f_i}, \cdot),$ $i =1,2$).
Poisson structures generalize symplectic structures by relaxing the non-degeneracy condition (which is algebraic), but still keeping the closedness condition\footnote{Note that for symplectic manifolds, the closedness of the symplectic form is equivalent with the Jacobi identity for the corresponding Poisson bracket.} (which is geometric). Thus Theorem \ref{main-thm} shows that it is the closedness condition which is crucial for the $C^0$-rigidity rather than non-degeneracy.   \\

It is known that in the symplectic setting the Poisson bracket exhibits $C^0$-rigid behaviour. 
The first result in this direction is due to F.~Cardin and C.~Viterbo \cite{Cardin-Viterbo} who proved that for all sequences $f_k, g_k \in C_c^{\infty}(M)$ which converge in the $C^0$ topology to smooth functions $f$ and $g,$ the following holds: if    
$\{f_k,g_k\} \ccon 0$ then $\{f,g\} =0.$ 
This result was improved later by M.~Entov, L.~Polterovich, and  F.~Zapolsky \cite{EP-Poisson,Zapolsky-poisson} where they proved that the Poisson bracket of a pair of functions is lower semicontinuous with respect to the $C^0$ topology, while L.~Buhovsky \cite{Buhovsky2/3} further improved this result by giving a sharp estimate on the rate of convergence.
Sufficient conditions for $\{f_k,g_k\}$ to converge to $\{f,g\}$ (in the $C^0$ topology) provided that $(f_k,g_k) \ccon (f,g),$ were given by V.~Humili\`ere \cite{Humiliere-poisson} and M.-C.~Arnaud \cite{Arnaud}.

The proof of Theorem \ref{main-thm} exploits a general principle that many rigidity phenomena in symplectic geometry arise from the energy-capacity inequality (see e.g.~\cite{Lalonde-McDuff}). 
More precisely, in the proof of Theorem \ref{main-thm} we use the following consequence of the energy-capacity inequality: 
if a sequence of Hamiltonian diffeomorphisms converges in both $C^0$ and Hofer topology then the limits coincide (see Proposition \ref{prop:hofer-c0-limit} below). For the definition of the Hofer norm for Poisson structures we refer the reader to Section \ref{sec:proof} or \cite{Joksimovic-Marcut}.

\subsection{Poisson homeomorphisms and some open questions}

The Eliashberg-Gromov theorem led to the definition of symplectic homeomorphisms as the $C^0$ limits of symplectic diffeomorphisms.
The behavior of symplectic homeomorphisms drew a lot of attention lately, as they play an important role
in symplectic geometry and dynamics (see e.g.~\cite{Sobhan-simplicity,Humiliere-Habilitaion,BO16,Oh-Muller} and references therein). 
Following the same principle we introduce the following.

 \begin{definition*}[Poisson homeomorphisms]   Let $(M,\pb)$ be a Poisson manifold.
    A map $\varphi \in \Homeo (M)$ is a \emph{Poisson homeomorphism} if there exists a sequence $\varphi_k \in \Poiss,$ $k \in \N$ such that $\varphi_k \ccon \varphi.$
\end{definition*}

From Theorem \ref{main-thm} it follows that a Poisson homeomorphism which is also a diffeomorphism is a Poisson diffeomorphism. 

The following general question naturally arises.

\begin{question*}
How much Poisson geometry do Poisson homeomorphisms remember?
\end{question*}

There are many interesting instances of the above question. Here we will state just some of them.

\begin{question}[Poisson homeomorphisms and symplectic leaves] \label{ques:symplectic-leaves}
    Do Poisson homeomorphisms map symplectic leaves to symplectic leaves homeomorphically? If so, does the restriction of a Poisson homeomorphism to a symplectic leaf induce a symplectic homeomorphism between the leaf and its image?
\end{question}

Another interesting problem is to understand how various classes of submanifolds behave with respect to Poisson homeomorphisms. It is known that symplectic homeomorphisms express very interesting behavior in this sense. For example, coisotropic submanifolds are preserved by symplectic homeomorphisms (see \cite{HLS} for a precise statement), while there is an example of a 
 symplectic homeomorphism which maps a symplectic submanifold to an isotropic submanifold (see \cite{BO16}).

Let $(M, \pb)$ be a Poisson manifold and $N \subseteq M.$ 
The \emph{vanishing ideal of $N$} is 
$$\II (N) := \{ f \in C^{\infty} (M) \hspace{1mm} \vert \hspace{1mm} f (x) = 0, \text{ } \forall x \in N \}.$$
A submanifold $N \subseteq (M, \pb)$ is \emph{coisotropic} if the vanishing ideal $\II (N)$ is a Lie subalgebra, i.e.~for every $f, g \in \II (N)$ it holds that $\{f,g\} \in \II (N).$ 
Following \cite{HLS} one could ask the following.

\begin{question}[Poisson coisotropic $C^0$-rigidity] \label{ques:coiso-rigidity}
Let $(M,\pb)$ be a Poisson manifold, $\varphi$ be a Poisson homeomorphism, and $N \subseteq M$ be a coisotropic submanifold such that $\varphi(N)$ is a smooth submanifold of $M.$ Is $\varphi(N)$ coisotropic? 
\end{question}

From an abstract point of view, Poisson manifolds can be seen as certain ``quotients'' of symplectic manifolds. In this sense it is natural to ask when Poisson homeomorphisms induce symplectic homeomorphisms of the associated symplectic manifold. 

To make this more precise let $(M, \pb)$ be a Poisson manifold. 
A \emph{symplectic realization} of $(M,\pb)$ is a pair $((S,\omega), \mu)$
where $(S, \omega)$ is a smooth symplectic manifold and $\mu: (S,\omega) \arr (M,\pb)$ is a surjective submersion which is a Poisson map.
In practice, constructed symplectic realizations $(S,\omega)$ have additional structure of a (symplectic) Lie groupoid. 

A Lie groupoid is a tuple $(\G, M, t,s,m,u,i)$ where $\G$ (the space of ``arrows'') and $M$ (the space of ``objects'') are smooth manifolds, $s,t: \G \arr M$ (the ``source'' and ``target'' maps) are smooth submersions, and $u: M \arr \G, i: \G \arr G, m: \G^{(2)} \arr G$ are smooth maps, where $\G^{(2)}$ denotes the space of so-called ``composable arrows''. 
For more details on (symplectic) Lie groupoids and symplectic realizations we refer to the book \cite{CMF_book}.

\begin{definition*}[Integrable Poisson manifolds]
 A Poisson manifold $(M,\pb)$ is called \emph{integrable} if there exists a symplectic groupoid $(\Sigma, \Omega)$ such that $((\Sigma, \Omega), t)$ is a symplectic realization of $(M,\pb)$ where $t$ is the target map and such that the canonical embedding (using the unit map) $u: M \arr \Sigma$ is a Lagrangian embedding.   
\end{definition*}

It is known that if $(M,\pb)$ is integrable then there is a unique integration $(\Sigma,\Omega)$ with simply-connected source-fibers and this one we call the \emph{canonical integration}. For more details on the problem of integrability of Poisson manifolds we refer to \cite{Crainic,Xu,CMF_book} and references therein.

Integration of Poisson manifolds can be seen as a generalization of the fact that for every smooth manifold there is a canonical way of associating a symplectic manifold by taking its cotangent bundle. It turns out that the cotangent bundles represent integrations of the trivial Poisson structures $\pb \equiv 0,$
where the target map $t$ is the canonical projection and the unit embedding $u$ is the canonical embedding $M \arr T^*M$ as the zero-section. For more examples we refer to \cite[Chapters 12-14]{CMF_book}.

Assume that $(M,\pb)$ is an integrable Poisson manifold and denote by $((\Sigma, \Omega),t)$ the corresponding (canonical) integration. One could ask whether a Poisson homeomorphism of $(M,\pb)$ induce a symplectic homeomorphisms on the integration $(\Sigma, \Omega).$

\begin{question} \label{ques:groupoid-homeo}
    Let $\varphi$ be a Poisson homeomorphism $(M, \pb).$ Does there exist a symplectic homeomorphism $\wt \varphi$ of the symplectic groupoid $(\Sigma, \Omega)$ such that $t \circ \varphi = \wt\varphi \circ t$?
\end{question}

This question is already interesting for the trivial Poisson structure since it is not clear whether a homeomorphism of the base induce a symplectic homeomorphism of the cotangent bundle. Note that a diffeomorphism of the base always induces a (smooth) symplectomorphism of the cotangent bundle in the above sense. Namely, for every $\varphi \in \Diff (M)$ the map given by 
$$T^*M \ni (x,p) \mapsto \left(\varphi(x), p (d \varphi^{-1} (\cdot)) \right) \in T^*M$$
is a symplectomorphism of $T^*M$ with respect to the canonical symplectic structure.

\begin{remark}[$C^0$ rigidity of Lagrangian bisections] \label{rmk:lagr-bisections} \normalfont
    Let $(\Sigma, \Omega)$ be a symplectic groupoid integrating a closed\footnote{compact and without boundary} Poisson manifold $(M,\pb).$ 
    A \emph{Lagrangian bisection} is a section\footnote{Here by section we mean with respect to the target map, i.e.~$t\circ b = \operatorname{id}_M: M \arr M$} $b: M \arr \Sigma$ such that $b^*\Omega = 0$ and such that $s \circ b: M \arr M$ is a diffeomorphism where $s$ is the source map.
    It is not hard to check that for a given Lagrangian bisection $b$ the induced diffeomorphism $s \circ b$ is a Poisson diffeomorphism. Such maps define a subgroup $\Gamma(\Sigma,\Omega)$ of $\Poiss.$ 
   
    By the Laudenbach-Sikorav's theorem \cite{Laudenbach-Sikorav} about the $C^0$-rigidity of Lagrangian embeddings it follows that the space of Lagrangian bisections is closed with respect to the $C^0$ topology, and hence the induced group $\Gamma(\Sigma,\Omega)$ is closed in the $C^0$ topology inside $\Diff (M)$.
    Note that in general $\Gamma(\Sigma,\Omega) \subsetneq \Poiss$ and hence Theorem \ref{main-thm} gives a stronger result.

    On the other hand, $C^0$ rigidity of Lagrangian bisections could carry some additional information (not necessarily related to the group $\Poiss$).
    For example, consider the case where $(M,\pb \equiv 0).$ Then $(\Sigma,\Omega) = (T^*M, \omega_{can})$ and the Lagrangian bisections are exactly closed 1-forms. Hence from the $C^0$-rigidity of Lagrangian bisections it follows that the space of closed differential 1-forms is closed inside the space of all differential 1-forms equipped with the $C^0$ topology. Note that in this case $\Gamma(\Sigma,\Omega) = \{\operatorname{id}_M \},$ and therefore the $C^0$ closedness of $\Gamma(\Sigma,\Omega)$ trivially follows. 

\end{remark}

\subsection*{Acknowledgement}
 I would like to thank Ioan M\u{a}rcu\cb{t} for carefully reading and providing a very useful feedback on a preliminary version of the article and for suggesting Questions \ref{ques:symplectic-leaves} and \ref{ques:groupoid-homeo}, and to
  Marius Crainic for an interesting discussion and suggesting Remark \ref{rmk:lagr-bisections}. 
The idea for the proof of Theorem \ref{main-thm} arose while I was preparing a talk for the UGC seminar at Utrecht University and therefore I would like to thank Fabian Ziltener and \'Alvaro del Pino G\'omez for inviting me to give a talk.

The work on this project was funded by Agence Nationale de la Recherche through ``ANR COSY: New challenges in contact and symplectic topology'' grant (decision ANR-21-CE40-0002).

\section{Proof of Theorem \ref{main-thm}} \label{sec:proof}

First, we recall the definition of the Hofer metric on $\Ham.$

Let $(M, \pb)$ be a Poisson manifold and $f \in C_c^{\infty}([0,1] \times M).$
The length of the Hamiltonian isotopy generated by $f$ is defined as
\[l(f):=\int_0^1\big(\sup_{x \in M} f_t(x)-\inf_{x \in M} f_t(x) \big) \ dt.\]
Notice that, contrary to the symplectic case, the 
length of the Hamiltonian isotopy depends on the choice of a Hamiltonian function that generates the isotopy since Casimir functions need not be constant in general.
We define a norm on $\mathrm{Ham}(M,\pb)$ by
\begin{equation} \label{eq:hofer-norm}
    \vv \fii \vvh:=\inf\big\{l(f)\ \vert\ f\in C^{\infty}_{c}([0,1]\times M),\ \varphi_f^1=\fii \big\},
\end{equation}
which we call the Hofer norm.

The Hofer norm was first introduced on symplectic manifolds.
It is not hard to check that \eqref{eq:hofer-norm} defines a conjugation invariant pseudo-norm on $\Ham,$ while
the proof of the non-degeneracy relies on hard methods from symplectic topology.

In the symplectic case, H.~Hofer \cite{Hofer} proved non-degeneracy for the standard symplectic structure on $\mathbb{R}^{2n}.$ Later L.~Polterovich \cite{Polterovich-hofer} extended it to a larger class of symplectic manifolds, and F.~Lalonde and D.~McDuff \cite{Lalonde-McDuff} provided a proof for all symplectic manifolds.  

In the Poisson setting the non-degeneracy of the Hofer norm was proven by I.~M\u{a}rcu\cb{t} and the author in \cite{Joksimovic-Marcut}, reducing the setting to the symplectic case by restricting to a symplectic leaf. Before that, the non-degeneracy was known for  Poisson manifolds whose symplectic leaves are closed embedded submanifolds due to
D.~Sun and Z.~Zhang \cite{Sun}\footnote{Actually, in 
\cite{Sun} the non-degeneracy of the Hofer norm was claimed for regular Poisson manifolds, but in the proof they do not use regularity, but the assumption that the restriction of a compactly supported function to a leaf is compactly supported, however, without stating this explicitly.}, and for Poisson manifolds whose closed leaves form a dense set due to T.~Rybicki \cite{Rybicki}.

We refer the reader to the book by L.~Polterovich \cite{Polterovich-hofer} for a detailed overview on Hofer geometry and to \cite{Burago-Ivanov-Polterovich} for a more general discussion on the importance of conjugation-invariant norms on various symmetry groups. \\

The main ingredient of the proof of Theorem \ref{main-thm} is the following proposition which is the Poisson analog of \cite[Prop.~3.6]{Oh-Muller}.

\begin{proposition} \label{prop:hofer-c0-limit}
    Let $\varphi_k \in \Ham,$ $k \in \N$ be a sequence such that $\varphi_k \hcon \varphi \in \Ham$ and $\varphi_k \ccon \psi \in \Homeo (M).$ Then $\varphi = \psi.$
\end{proposition}

The proof uses the Poisson version of the \emph{energy-capacity inequality}: 
let $\varphi \in \Ham$ and $B \subseteq L$ be an open ball such that $\varphi(B) \cap B = \emptyset.$ Then we have
\begin{equation} \label{eq:energy-capacity}
    \vv \varphi \vvh \geq e(B,L),
\end{equation}
where $e(B,L)$ is the displacement energy of $B$ inside the symplectic manifold $L.$  The proof follows from the standard energy-capacity inequality after restricting to a symplectic leaf, see \cite{Joksimovic-Marcut} for the details.

\begin{proof}[Proof of Proposition \ref{prop:hofer-c0-limit}]
Assume on the contrary that $\varphi \neq \psi.$ Then there exists $x \in M$ such that $\varphi^{-1}\psi (x) \neq x.$ Denote by $L$ the symplectic leaf through $x.$ Then there exists an open ball $B \subseteq L$ which contains $x$ and which is displaced by $\varphi^{-1}\psi.$ Hence there exists $k_0 \in \N$ such that for every $k \geq k_0$ it holds that $$\varphi^{-1}\varphi_k (B) \cap B = \emptyset.$$
Note that Hamiltonian isotopies preserve symplectic leaves and hence $\varphi^{-1}\varphi_k (B) \subseteq L.$
Now, from the energy-capacity inequality \eqref{eq:energy-capacity} we have that
$$\vv \varphi^{-1} \varphi_k \vvh \geq e(B,L).$$
Since the displacement energy of an open set of a symplectic manifold is always positive,  
we get a contradiction with the fact that $\vv \varphi^{-1} \varphi_k \vvh \arr 0.$
Hence $\varphi = \psi.$ This completes the proof of Proposition \ref{prop:hofer-c0-limit}.
\end{proof} 

\begin{corollary} \label{cor:flows-conjugation}
Let $\psi_k \in \Poiss,$ $k \in \N$ be a sequence which converges in $C^0$ topology to a smooth map $\psi \in \Diff (M).$ Then for every $f \in C_c^{\infty}(M)$ it holds that $\varphi_{f \circ \psi}^t = \psi^{-1} \varphi_f^t \psi,$ $\forall t \in [0,1].$
\end{corollary}

\begin{proof}
    It is enough to prove the statement for $t=1,$ then the other cases follow after rescaling Hamiltonians.
    Note that $f \circ \psi_k \ccon f \circ \psi$ implies $\varphi_{f \circ \psi_k}^1 \hcon \varphi_{f \circ \psi}^1.$ 
    Namely, denoting $F:= f \circ \psi,$ $F_k :=f \circ \psi_k$ we have that the Hamiltonian
    $$F \#  \Bar{F}_k := F  - F_k \circ \varphi_F^{-t} = (F - F_k) \circ \varphi_F^{-t},$$
    generates the flow $(\varphi^t_{f\circ \psi} \circ \varphi_{f \circ \psi_k}^{-t})_{t \in [0,1]},$ and therefore
    $$\vv \varphi^1_{f\circ \psi} \circ \varphi_{f \circ \psi_k}^{-1}\vvh \leq l(F \# \Bar{F}_k) \leq 2\vv f \circ \psi -f \circ \psi_k \vv_{\infty} \overset{k \arr \infty}{\longrightarrow} 0 .$$
    On the other hand 
    \begin{equation*} 
        \varphi^1_{f \circ \psi_k} = \psi_k^{-1} \varphi^1_f \psi_k \ccon \psi^{-1} \varphi^1_f \psi.
    \end{equation*}
    Now Corollary \ref{cor:flows-conjugation} follows from Proposition \ref{prop:hofer-c0-limit} applied to the sequence $\{\varphi^1_{f \circ \psi_k}\}_{k \in \N}.$
    %
\end{proof}

We are now ready for the proof of the main result.

\begin{proof}[Proof of Theorem \ref{main-thm}]
Let $\psi_k \in \Poiss, k \in \N$ be a sequence which converges in the $C^0$ topology to a smooth map $\psi \in \Diff (M).$ We will show that $\psi$ preserves the Poisson bracket of all compactly supported functions $f,g \in C_c^{\infty} (M),$ see Remark \ref{rmk:comp_supp}.

Let $f,g \in C_c^{\infty}(M).$ We define $F,F_k,G,G_k \in C_c^{\infty}(M)$ by
\begin{align*}
    F_k := f \circ \psi_k,& \quad F := f\circ \psi,\\
    G_k := g \circ \psi_k,& \quad G := g \circ \psi.
\end{align*}
Notice that $F_k \ccon F,$ $G_k \ccon G.$ 
From Corollary \ref{cor:flows-conjugation} we get that 
\begin{equation}\label{eq:gk_flow}
    \varphi^t_{G_k} = \psi_k^{-1} \varphi_g^t \psi_k  \ccon  \psi^{-1} \varphi_g^t \psi = \varphi_G^t, \quad \forall t \in [0,1].
\end{equation}
Therefore
\begin{align*} 
    F - F \circ \varphi_{G}^t &= \lim_{k \arr \infty} F_k - F_k \circ \varphi_{G_k}^t \\ \nonumber
                        &= \lim_{k \arr \infty} \int_0^t \{ F_k, G_k\} \circ \varphi_{G_k}^s ds \\
                        &= \lim_{k \arr \infty} \int_0^t \{ f \circ \psi_k, g \circ \psi_k\} \circ \varphi_{G_k}^s ds. 
\end{align*}
Using that $\psi_k \in \Poiss$ we get
\begin{equation*}
     F - F \circ \varphi_{G}^t = \lim_{k \arr \infty} \int_0^t \{ f,g\} \circ \psi_k \circ \varphi_{G_k}^s ds,
\end{equation*}
and hence from \eqref{eq:gk_flow} and the fact that $\psi_k \ccon \psi$ it follows that
\begin{equation}\label{eq:lhs}
    F - F \circ \varphi_{G}^t = \int_0^t \{ f,g\} \circ \psi \circ \varphi_{G}^s ds.
\end{equation}
On the other hand 
\begin{equation} \label{eq:rhs}
    F - F \circ \varphi_{G}^t = \int_0^t \{F, G\} \circ \varphi_{G}^s ds =  \int_0^t \{f \circ \psi, g \circ \psi\} \circ \varphi_{G}^s ds.
\end{equation}
Differentiating both expressions \eqref{eq:lhs} and \eqref{eq:rhs} with respect to $t,$ and setting $t=0,$ we get that 
$$\{f \circ \psi,  g\circ \psi\} = \{f,g\} \circ \psi,$$ for all $f,g \in C_c^{\infty}(M).$
This completes the proof of Theorem \ref{main-thm}.

\end{proof}

\subsection*{Conflict of interest} On behalf of all authors, the corresponding author states that there is no conflict of
interest.

\bibliographystyle{alpha}


{\small
\bibliography{References}}
\vspace{1cm}

{\small

\medskip
 \noindent Du\v{s}an Joksimovi\'c\\
\noindent Universit\'e Paris-Saclay, F-91405 Orsay Cedex, France\\
 {\it e-mail:} dusan.joksimovic@universite-paris-saclay.fr

}

\end{document}